\documentclass{amsart}
\usepackage{amsfonts}
\usepackage{bbm}
\usepackage{mathrsfs}
\usepackage{amsmath}
\usepackage{bigdelim}
\usepackage{multirow}
\usepackage{amssymb}
\usepackage{indentfirst}
\usepackage{CJK}
\usepackage{fancyhdr}
\usepackage{amscd,amssymb,amsmath,graphicx,verbatim,xy}
\usepackage[TS1,OT1,T1]{fontenc}

\newtheorem{theorem}{Theorem}[section]
\newtheorem{lemma}[theorem]{Lemma}
\newtheorem{corollary}[theorem]{Corollary}

\theoremstyle{definition}

\theoremstyle{remark}
\newtheorem{remark}[theorem]{Remark}
\numberwithin{equation}{section}

\begin{document}

\title[A new version of the Gelfand-Hille theorem]{A new version of the Gelfand-Hille theorem}

\author{Junsheng Fang}
\address{Junsheng Fang, School of Mathematics, Hebei Normal University, 050016, Shijiazhuang, P. R. China}
\email{jfang@hebtu.edu.cn}

\author{Bingzhe Hou}
\address{Bingzhe Hou, School of Mathematics, Jilin University, 130012, Changchun, P. R. China}
\email{houbz@jlu.edu.cn}
	
\author{Chunlan Jiang}
\address{Chunlan Jiang, School of Mathematics, Hebei Normal University, 050016, Shijiazhuang, P. R. China}
\email{cljiang@hebtu.edu.cn}
\date{}
\thanks{}
\begin{abstract}
Let $\mathcal{X}$ be a complex Banach space and $A\in\mathcal{L}(\mathcal{X})$ with $\sigma(A)=\{1\}$. We prove that for a vector $x\in \mathcal{X}$, if $\|(A^{k}+A^{-k})x\|=O(k^N)$ as $k \rightarrow +\infty$ for some positive integer $N$, then $(A-\mathbf{I})^{N+1}x=0$ when $N$ is even and $(A-\mathbf{I})^{N+2}x=0$ when $N$ is odd. This could be seemed as a new version of the Gelfand-Hille theorem. As a corollary, we also obtain that for a quasinilpotent operator $Q\in\mathcal{L}(\mathcal{X})$ and a vector $x\in\mathcal{X}$, if $\|\cos(kQ)x\|=O(k^N)$ as $k \rightarrow +\infty$ for some positive integer $N$, then $Q^{N+1}x=0$ when $N$ is even and $Q^{N+2}x=0$ when $N$ is odd.
\end{abstract}
\maketitle

\section{Introduction}

Let $\mathcal{X}$ be a complex Banach space and $A\in\mathcal{L}(\mathcal{X})$ be a bounded linear operator whose spectrum $\sigma(A)$ is the singleton $\{1\}$. Let $r$ be a positive integer. The Gelfand-Hille theorem asserts that $(A-\mathbf{I})^r=0$ if and only if $\|A^n\|=O(|n|^{r-1})$ or $\|A^n\|=o(|n|^{r})$ as $|n|\rightarrow \infty$, where $\mathbf{I}$ means the identity operator.  This result was firstly given by Gelfand \cite{Gel41} in 1941 for $r=1$, and then Hille \cite{Hille44} proved the general case in 1944 (see also \cite{Stone48} of Stone). Furthermore, Shilov \cite{Shi50} in 1950 gave an example to show that the condition only for positive powers of $A$ is not sufficient for the Gelfand-Hille theorem. In 1952, Wermer \cite{Wer52} used this theorem to prove that an invertible operator $T$ on a Banach space $\mathcal{X}$ has a nontrivial invariant subspace if $\|T^n\|=O(|n|^{r})$ as $|n|\rightarrow \infty$ for some positive integer $r$.

So far, the Gelfand-Hille theorem has been generalized to some different versions and has been applied to the operator theory such as Ces\`{a}ro mean and invariant subspaces problem, we refer to the survey \cite{Zem1994} of Zem\'{a}nek and the references therein. One of the local versions of the Gelfand-Hille theorem given by Aupetit and Drissi \cite{Aup94} asserts that for $A\in\mathcal{L}(\mathcal{X})$ with $\sigma(A)=\{1\}$ and $x\in\mathcal{X}$, if $\|A^nx\|=O(|n|^{r})$ as $|n|\rightarrow \infty$, then $(A-\mathbf{I})^{r+1}x=0$. Another local version given by Atzmon \cite{Atz80} shows that if $\|A^nx\|=O(n^{r})$ as $n\rightarrow +\infty$ and $n\|(A-\mathbf{I})^{n}x\|^{\frac{1}{n}}=0$ as $n\rightarrow +\infty$, then $(A-\mathbf{I})^{r+1}x=0$. Obviously, local version of the Gelfand-Hille theorem implies the original version.

In the present paper, we will give a generalization of the local Gelfand-Hille theorem. Our main theorem asserts that for $A\in\mathcal{L}(\mathcal{X})$ with $\sigma(A)=\{1\}$ and $x\in\mathcal{X}$, if $\|(A^{k}+A^{-k})x\|=O(k^N)$ as $k \rightarrow +\infty$, then
\[
\begin{matrix}
(A-\mathbf{I})^{N+1}x=0 & \text{if} \ N \ \text{is even}, \\
(A-\mathbf{I})^{N+2}x=0 & \text{if} \ N \ \text{is odd}.
\end{matrix}
\]
As a corollary, we also obtain that for a quasinilpotent operator $Q\in\mathcal{L}(\mathcal{X})$ and a vector $x\in\mathcal{X}$, if $\|\cos(kQ)x\|=O(k^N)$ as $k \rightarrow +\infty$ for some positive integer $N$, then
\[
\begin{matrix}
Q^{N+1}x=0 & \text{if} \ N \ \text{is even}, \\
Q^{N+2}x=0 & \text{if} \ N \ \text{is odd}.
\end{matrix}
\]

\section{A lemma on finite dimensional spaces}

Let $A\in\mathcal{L}(\mathcal{X})$ with $\sigma(A)=\{1\}$ and $x\in\mathcal{X}$. As well-known, $A-\mathbf{I}$ is nilpotent if $\mathcal{X}$ is a finite dimensional space. In this section, we will give a lemma to describe the growth of $\|(A^k+A^{-k})x\|$. Following from the Jordan decomposition theorem, each linear operator acting on a finite dimensional space is similar to a direct sum of some Jordan blocks. A Jordan block is of the form
\[
J_d(\lambda)=\begin{bmatrix}
\lambda & 1 &  & & 0 \\
  & \lambda &1  & &  \\
  &  & \ddots & \ddots  & \\
 &  & &\lambda &1 \\
0 &  & & &\lambda
\end{bmatrix}_{d\times d},
\]
where $d\in\mathbb{N}$ and $\lambda\in\mathbb{C}$. So it sufficed to consider the Jordan blocks, if the space is of finite dimension.

Before introducing the lemma, let us review some notions and basic results.
A nonzero vector $x$ in a complex Banach space $\mathcal{X}$ is called a cyclic vector of a linear operator $T$, if $\textrm{span}\{T^{k}x;~k=0,1,\ldots\}=\mathcal{X}$, where $\textrm{span}\{T^{k}x;~k=0,1,\ldots\}$ is the norm closure of the linear span of elements in $\{T^{k}x\}_{k=0}^{\infty}$. Denote
\[
\mathfrak{M}_x\triangleq\textrm{span}\{T^{k}x; \ k=0,1,2, \ldots\}.
\]
Then, $\mathfrak{M}_x$ is an invariant subspace of $T$. If $\mathfrak{M}_x$ is of finite dimension, then $T|_{\mathfrak{M}_x}$ is similar to a Jordan block. Moreover, a vector $x$ is cyclic under the action of a Jordan block if and only if the last coordinate of $x$ is nontrivial.

Let $\{a_n\}_{n=0}^{\infty}$ and $\{b_n\}_{n=0}^{\infty}$ be two sequence of positive numbers. If there exists $M>0$ such that $a_n\leq Mb_n$ for all $n$, we call that $a_n$ is quasi-dominated by $b_n$ and denote $a_n\precsim b_n$. Furthermore, if $a_n$ and $b_n$ are mutually quasi-dominated, we call they are equivalent and denote $a_n\sim b_n$. In particular, the usual notation $a_n=O(n^p)$ means that $a_n$ is quasi-dominated by $n^p$.

\begin{lemma}\label{Jn}
Let $\mathbb{C}^{d+1}$ be the $(d+1)$-dimension complex linear normed space. Let $J_{d+1}(1)$ be the $(d+1)$-Jordan block acting on $\mathbb{C}^{d+1}$. Then,
for any $x\in\mathbb{C}^{d+1}$,
\begin{enumerate}
  \item $\|J_{d+1}(1)^kx\|\precsim k^{d}$.
  \item When $d$ is even, $\|(J_{d+1}(1)^k+J_{d+1}(1)^{-k})x\|\precsim k^{d}$.
  \item When $d$ is odd, $\|(J_{d+1}(1)^k+J_{d+1}(1)^{-k})x\|\precsim k^{d-1}$.
\end{enumerate}
Moreover, in each of $(1-3)$, the $'\sim'$ holds if and only if $x$ is a cyclic vector of $J_{d+1}(1)$.
\end{lemma}

\begin{proof}
Notice that $J_{d+1}(1)=\mathbf{I}+J_{d+1}(0)$ and $J_{d+1}(0)^{d+1}=0$. Then for any positive integer $k$ larger than $d$,
\begin{align*}
J_{d+1}(1)^k&=(\mathbf{I}+J_{d+1}(0))^k=\sum\limits_{j=0}^{k}C_k^jJ_{d+1}(0)^j=\sum\limits_{j=0}^{d}C_k^jJ_{d+1}(0)^j \\
&=\begin{bmatrix}
1 & C_k^1 &C_k^2  &\cdots & C_k^d \\
0  & 1 &C_k^1  &\cdots &C_k^{d-1}  \\
\vdots  &\vdots  & \ddots & \ddots  &\vdots \\
\vdots &\vdots  & &\ddots &C_k^1 \\
0 &0  &\cdots &\cdots &1
\end{bmatrix},
\end{align*}
where $C_k^j=\frac{k!}{j!(k-j)!}$.
Similarly, for $k\geq d$,
\begin{align*}
J_{d+1}(1)^{-k}&=(\mathbf{I}+J_{d+1}(0))^{-k}=\sum\limits_{j=0}^{d}(-1)^jC_{k+j-1}^jJ_{d+1}(0)^j \\
&=\begin{bmatrix}
1 & -C_k^1 &C_{k+1}^2  &\cdots & (-1)^dC_{k+d-1}^d \\
0  & 1 &-C_k^1  &\cdots &(-1)^{d-1}C_{k+d-2}^{d-1}  \\
\vdots  &\vdots  & \ddots & \ddots  &\vdots \\
\vdots &\vdots  & &\ddots &-C_k^1 \\
0 &0  &\cdots &\cdots &1
\end{bmatrix}
\end{align*}
Moreover, for $k\geq d$,
\begin{align*}
J_{d+1}(1)^{k}+J_{d+1}(1)^{-k}&=\sum\limits_{j=0}^{d}\left(C_k^j+(-1)^jC_{k+j-1}^jJ_{d+1}(0)^j\right) \\
&=\begin{bmatrix}
2 & 0 &C_{k}^2+C_{k+1}^2  &\cdots & C_k^d+(-1)^dC_{k+d-1}^d \\
0  & 2 &0  &\cdots &C_k^{d-1}+(-1)^{d-1}C_{k+d-2}^{d-1}  \\
\vdots  &\vdots  & \ddots & \ddots  &\vdots \\
\vdots &\vdots  & &\ddots &0 \\
0 &0  &\cdots &\cdots &2
\end{bmatrix}
\end{align*}
Then, for any $x\in\mathbb{C}^{d+1}$,
\[
\|J_{d+1}(1)^kx\|\precsim C_k^d\sim k^{d},
\]
and
\[
\|(J_{d+1}(1)^k+J_{d+1}(1)^{-k})x\|\precsim 2+C_k^d+(-1)^dC_{k+d-1}^d \sim
\left\{\begin{matrix}
k^{d} & \text{if} \ N \ \text{is even}, \\
k^{d-1} & \text{if} \ N \ \text{is odd}.
\end{matrix}\right.
\]
It is not difficult to see that the $'\sim'$ holds if and only if the last coordinate of $x$ is nontrivial, which is equivalent to $x$ being cyclic.
\end{proof}

\section{Main results}

Our main theorem is as follows.

\begin{theorem}\label{newGH}
Let $\mathcal{X}$ be a complex Banach space, $x\in\mathcal{X}$ and $A\in\mathcal{L}(\mathcal{X})$ with $\sigma(A)=\{1\}$. If $\|(A^{k}+A^{-k})x\|=O(k^N)$ as $k \rightarrow +\infty$ for some positive integer $N$, then
\[
\begin{matrix}
(A-\mathbf{I})^{N+1}x=0 & \text{if} \ N \ \text{is even}, \\
(A-\mathbf{I})^{N+2}x=0 & \text{if} \ N \ \text{is odd}.
\end{matrix}
\]
\end{theorem}

\begin{proof}
Let
\[
\mathfrak{M}_x\triangleq\textrm{span}\{A^{k}x; \ k=0,1,2, \ldots\}.
\]
Since
\[
A^{-1}=(\mathbf{I}-(\mathbf{I}-A))^{-1}=\sum\limits_{n=0}^{\infty}(\mathbf{I}-A)^n,
\]
one can see that
\[
\mathfrak{M}_x=\textrm{span}\{A^{k}x; \ k=0,\pm1, \pm2, \ldots\}
\]
and for any $n\in\mathbb{N}$
\[
\mathfrak{M}_x=\textrm{span}\{A^{k}x; \ k=n,n+1,n+2, \ldots\}.
\]
Moreover, we also have $\sigma(A|_{\mathfrak{M}_x})=\{1\}$ and then $\sigma(A^{-1}|_{\mathfrak{M}_x})=\{1\}$.

Let
\[
\mathfrak{N}\triangleq\textrm{span}\{(A^{k}+A^{-k})x; \ k=0,1,2, \ldots\}\subseteq\mathfrak{M}_x.
\]
Firstly, we will prove that $\mathfrak{N}$ is a finite dimensional space.
Denote by $R(\lambda, A)=(\lambda-A)^{-1}$ the resolvent function. Then the operator valued function $R(\lambda, A)$ is analytic on $\mathbb{C}\setminus\{1\}$ and has a Laurent expansion
\[
R(\lambda, A)=(\lambda-A)^{-1}=\sum\limits_{k=0}^{\infty}\frac{(A-\mathbf{I})^{k}}{(\lambda-1)^{k+1}}.
\]
Given any $f^*\in \mathfrak{M}_x^*$, where $\mathfrak{M}_x^*$ is the dual space of $\mathfrak{M}_x$, such that for any $k=0,1, \ldots, N+1$,
\[
f^*((A-\mathbf{I})^{k}x)=0 \ \ \ \text{and} \ \ \  f^*((A^{-1}-\mathbf{I})^{k}x)=0.
\]
Then $F(\lambda)=f^*(R(\lambda, A)x+R(\lambda, A^{-1})x)$ is analytic on $\mathbb{C}\setminus\{1\}$ and has a Laurent expansion
\begin{align*}
F(\lambda)&=f^*(R(\lambda, A)x+R(\lambda, A^{-1})x) \\
&=\sum\limits_{k=0}^{\infty}\frac{f^*((A-\mathbf{I})^{k}x+(A^{-1}-\mathbf{I})^{k}x)}{(\lambda-1)^{k+1}} \\
&=\sum\limits_{k=N+2}^{\infty}\frac{f^*((A-\mathbf{I})^{k}x+(A^{-1}-\mathbf{I})^{k}x)}{(\lambda-1)^{k+1}}.
\end{align*}
Let
\[
G(\lambda)=(-1)^{N+2}\sum\limits_{k=0}^{\infty}\frac{f^*((A-\mathbf{I})^{k+N+2}x+(A^{-1}-\mathbf{I})^{k+N+2}x)}{(k+1)(k+2)\cdots(k+N+2)(\lambda-1)^{k+1}}.
\]
Then $G(\lambda)$ is also an analytic function on $\mathbb{C}\setminus\{1\}$. Moreover, $F(\lambda)$ is the $(N+2)$-th derivation of $G(\lambda)$, i.e.,
\[
G^{(N+2)}(\lambda)=F(\lambda).
\]
On the open unit disk $|\lambda|<1$,
\[
F(\lambda)=f^*(R(\lambda, A)x+R(\lambda, A^{-1})x)=\sum\limits_{k=0}^{\infty}{f^*(A^{-(k+1)}x+A^{k+1}x)}\cdot{\lambda^{k}}.
\]
Then, on the open unit disk $|\lambda|<1$,
\[
G(\lambda)=P_{N+1}(\lambda)+\sum\limits_{k=0}^{\infty}\frac{f^*(A^{-(k+1)}x+A^{k+1}x)}{(k+1)(k+2)\cdots(k+N+2)}\cdot\lambda^{k+N+2},
\]
where $P_{N+1}(\lambda)$ is some certain $(N+1)$-order polynomial. Following from $\|(A^{k}+A^{-k})x\|=O(k^N)$, the sequence of the Taylor coefficients of $G(\lambda)$ belongs to $\ell^{1}$ and consequently $G(\lambda)$ is bounded on $\overline{\mathbb{D}}\setminus\{1\}$.

Similarly, on the annulus $|\lambda|>1$,
\begin{align*}
F(\lambda)&=f^*(R(\lambda, A)x+R(\lambda, A^{-1})x) \\
&=\sum\limits_{k=0}^{\infty}\frac{f^*(A^{k}x+A^{-k}x)}{\lambda^{k+1}} \\
&=\sum\limits_{k=N+2}^{\infty}\frac{f^*(A^{k}x+A^{-k}x)}{\lambda^{k+1}}.
\end{align*}
Then, on the annulus $|\lambda|>1$,
\begin{align*}
G(\lambda)&=\sum\limits_{k=N+2}^{\infty}\frac{f^*(A^{k}x+A^{-k}x)}{(k-N-1)(k-N)\cdots k\cdot\lambda^{k-N-1}} \\
&=\sum\limits_{k=0}^{\infty}\frac{f^*(A^{k+N+2}x+A^{-(k+N+2)}x)}{(k+1)(k+2)\cdots(k+N+2)\lambda^{k+1}}.
\end{align*}
Following from $\|(A^{k}+A^{-k})x\|=O(k^N)$, the sequence of the Laurent coefficients of $G(\lambda)$ on $|\lambda|>1$ belongs to $\ell^{1}$ and consequently $G(\lambda)$ is bounded on $\mathbb{C}\setminus\overline{\mathbb{D}}$. Thus, the function $G(\lambda)$ is an entire function. In addition to $\lim\limits_{\lambda\rightarrow\infty}G(\lambda)=0$, we have $G(\lambda)\equiv 0$. So far, we obtained that $f^*(A^{k}x+A^{-k}x)=0$ for any $0\leq k\leq N+1$ implies $f^*(A^{k}x+A^{-k}x)=0$ for all $k\in\mathbb{N}$. Therefore, $\mathfrak{N}$ is a subspace of finite (at most $N+2$) dimension.

Let $T=\frac{A+A^{-1}}{2}$. Notice that $\sigma(T)=\{1\}$ and
\[
\textrm{span}\{T^{k}x; \ k=0,1,2, \ldots\}=\textrm{span}\{(A^{k}+A^{-k})x; \ k=0,1,2, \ldots\}=\mathfrak{N}.
\]
Then the operator $T$ acting on the finite dimensional space $\mathfrak{N}$ is similar to the Jordan block $J_{\mathrm{dim}\mathfrak{N}}(1)$, and consequently
\[
(T-\mathbf{I})^{\mathrm{dim}\mathfrak{N}}x=0.
\]
Since $\frac{A^{-1}}{2}$ is invertible and
\[
T-\mathbf{I}=\frac{A+A^{-1}}{2}-\mathbf{I}=\frac{A^{-1}}{2}(A-\mathbf{I})^2,
\]
we have
\[
(A-\mathbf{I})^{2\mathrm{dim}\mathfrak{N}}x=0,
\]
which implies that $\mathfrak{M}_x$ is also a finite dimensional subspace.
Since $x$ is a cyclic vector of $A$ acting on $\mathfrak{M}_x$, the operator $A$ acting on the space $\mathfrak{M}_x$ is similar to the Jordan block $J_{\mathrm{dim}\mathfrak{M}_x}(1)$.

Therefore, by Lemma \ref{Jn}, we obtain that
\[
\begin{matrix}
(A-\mathbf{I})^{N+1}x=0 & \text{if} \ N \ \text{is even}, \\
(A-\mathbf{I})^{N+2}x=0 & \text{if} \ N \ \text{is odd}.
\end{matrix}
\]
\end{proof}

\begin{remark}\label{onetwo}
Obviously, The simultaneous establishment of  $\|A^{k}x\|=O(k^N)$ and $\|A^{-k}x\|=O(k^N)$ implies the establishment of  $\|(A^{k}+A^{-k})x\|=O(k^N)$. But the converse is not natural. The above theorem shows that the converse is also true. By the way, the manner in the above proof is also valid to prove the local version of Gelfand-Hille theorem given by Aupetit and Drissi.
\end{remark}

\begin{corollary}\label{cosnV}
Let $\mathcal{X}$ be a complex Banach space and let $Q\in\mathcal{L}(\mathcal{X})$ be a quasinilpotent operator. For a vector $x\in X$, if $\|\cos(kQ)x\|=O(k^N)$ as $k \rightarrow +\infty$ for some positive integer $N$, then
\[
\begin{matrix}
Q^{N+1}x=0 & \text{if} \ N \ \text{is even}, \\
Q^{N+2}x=0 & \text{if} \ N \ \text{is odd}.
\end{matrix}
\]
\end{corollary}
\begin{proof}
Let $A=\mathrm{e}^{\mathbf{i}Q}$. Then $\sigma(A)=\{1\}$ and
\[
A^{k}+A^{-k}=\mathrm{e}^{\mathbf{i}kQ}+\mathrm{e}^{-\mathbf{i}kQ}=\cos(kQ).
\]
By Theorem \ref{newGH}, one can see that
\[
\begin{matrix}
(A-\mathbf{I})^{N+1}x=0 & \text{if} \ N \ \text{is even}, \\
(A-\mathbf{I})^{N+2}x=0 & \text{if} \ N \ \text{is odd}.
\end{matrix}
\]
Notice that
\[
Q=-\mathbf{i}\ln A=-\mathbf{i}\ln (\mathbf{I}-(\mathbf{I}-A))=-\mathbf{i}\sum\limits_{n=1}^{\infty}\frac{(\mathbf{I}-A)^n}{n}.
\]
Then, if $N$ is even,
\[
Q^{N+1}x=(-\mathbf{i}\sum\limits_{n=1}^{\infty}\frac{(\mathbf{I}-A)^n}{n})^{N+1}x=(-\mathbf{i}\sum\limits_{n=1}^{\infty}\frac{(\mathbf{I}-A)^{n-1}}{n})^{N+1}(\mathbf{I}-A)^{N+1}x=0.
\]
Similarly, if $N$ is odd,
\[
Q^{N+2}x=(-\mathbf{i}\sum\limits_{n=1}^{\infty}\frac{(\mathbf{I}-A)^{n-1}}{n})^{N+2}(\mathbf{I}-A)^{N+2}x=0.
\]
\end{proof}
\section*{Acknowledgement}

\noindent The third author was supported by National Natural Science Foundation of China (Grant No. 11831006, 11920101001 and 11771117). 

\section*{Data availability}

\noindent Data sharing is not applicable to this article as no datasets were generated or analyzed during the current study.

\section*{Competing interests}

\noindent The author declares that there is no conflict of interest or competing interest.

\end{document}